\documentclass[12pt]{amsart}
\usepackage{amsmath, amssymb, amsthm, amsfonts, enumerate, verbatim}
\usepackage[all]{xy}
\usepackage{tikz}
\tikzstyle{punkt}=[circle, fill=black, minimum size=1mm,inner sep=0pt, draw]
\def\NZQ{\mathbb}               

\def\ZZ{{\NZQ Z}}

\def\frk{\frak}               

\def\Phi{{\frk n}}
\def\Phi{{\frk N}}
\def\MB{{\mathcal B}}

\def\MT{{\mathcal T}}

\def\MS{{\mathcal S}}

\def\MM{{\mathcal M}}
\def\KK{{\mathbb K}}

\def\opn#1#2{\def#1{\operatorname{#2}}} 
\opn\chara{char}
\opn\length{\ell}
\opn\pd{pd}
\opn\rk{rk}
\opn\projdim{proj\,dim}
\opn\injdim{inj\,dim}
\opn\rank{rank}
\opn\depth{depth}
\opn\grade{grade}
\opn\height{height}
\opn\embdim{emb\,dim}
\opn\codim{codim}

\opn\Tr{Tr}
\opn\bigrank{big\,rank}
\opn\superheight{superheight}
\opn\lcm{lcm}
\opn\trdeg{tr\,deg}
\opn\reg{reg}
\opn\lreg{lreg}
\opn\ini{in}
\opn\lpd{lpd}
\opn\size{size}
\opn\bigsize{bigsize}
\opn\cosize{cosize}
\opn\bigcosize{bigcosize}
\opn\sdepth{sdepth}
\opn\sreg{sreg}
\opn\link{link}
\opn\fdepth{fdepth}
\opn\lin{lin}
\opn\ini{in}
\opn\inm{inm}
\opn\div{div}
\opn\Div{Div}
\opn\cl{cl}
\opn\Cl{Cl}
\opn\Spec{Spec}
\opn\Supp{Supp}
\opn\supp{supp}
\opn\Sing{Sing}
\opn\Ass{Ass}
\opn\Min{Min}
\opn\Mon{Mon}
\opn\dstab{dstab}
\opn\astab{astab}
\opn\Syz{Syz}
\opn\Ann{Ann}
\opn\Rad{Rad}
\opn\Soc{Soc}
\opn\Im{Im}
\opn\Ker{Ker}
\opn\Coker{Coker}
\opn\Am{Am}
\opn\Hom{Hom}
\opn\Tor{Tor}
\opn\Ext{Ext}
\opn\End{End}
\opn\Aut{Aut}
\opn\id{id}

\opn\nat{nat}
\opn\pff{pf}
\opn\Pf{Pf}
\opn\GL{GL}
\opn\SL{SL}
\opn\mod{mod}
\opn\ord{ord}
\opn\Gin{Gin}
\opn\Hilb{Hilb}
\opn\sort{sort}
\opn\initial{init}
\opn\ende{end}
\opn\height{height}
\opn\type{type}
\opn\mdeg{mdeg}
\opn\aff{aff}
\opn\con{conv}
\opn\relint{relint}
\opn\st{st}
\opn\lk{lk}
\opn\cn{cn}
\opn\core{core}
\opn\vol{vol}
\opn\link{link}
\opn\star{star}
\opn\lex{lex}
\opn\sign{sign}
\opn\gr{gr}

\def\pot#1#2{#1[\kern-0.28ex[#2]\kern-0.28ex]}

\opn\dirlim{\underrightarrow{\lim}}
\opn\inivlim{\underleftarrow{\lim}}
\def\Implies{\ifmmode\Longrightarrow \else
	\unskip${}\Longrightarrow{}$\ignorespaces\fi}
\def\implies{\ifmmode\Rightarrow \else
	\unskip${}\Rightarrow{}$\ignorespaces\fi}
\def\iff{\ifmmode\Longleftrightarrow \else
	\unskip${}\Longleftrightarrow{}$\ignorespaces\fi}

\let\:=\colon
\newtheorem{Theorem}{Theorem}[section]
\newtheorem{Lemma}[Theorem]{Lemma}
\newtheorem{Corollary}[Theorem]{Corollary}
\newtheorem{Proposition}[Theorem]{Proposition}
\newtheorem{Remark}[Theorem]{Remark}

\newtheorem{Definition}[Theorem]{Definition}

\let\epsilon\varepsilon
\let\kappa=\varkappa
\def\pnt{{\raise0.5mm\hbox{\large\bf.}}}

\begin{document}
	\title{Induced matchings in Strongly biconvex graphs and some  algebraic applications}
	\author {Sara Saeedi Madani and Dariush Kiani}

	\address{Sara Saeedi Madani, Faculty of Mathematics and Computer Science, Amirkabir University of Technology (Tehran Polytechnic), Tehran, Iran, and School of Mathematics, Institute for Research in Fundamental Sciences (IPM), Tehran, Iran}
	\email{sarasaeedi@aut.ac.ir}

	\address{Dariush Kiani, Faculty of Mathematics and Computer Science, Amirkabir University of Technology (Tehran Polytechnic), Tehran, Iran, and School of Mathematics, Institute for Research in Fundamental Sciences (IPM), Tehran, Iran}
	\email{dkiani@aut.ac.ir, dkiani7@gmail.com}

	\begin{abstract}
In this paper, motivated by a question posed in \cite{AH}, we introduce strongly biconvex graphs as a subclass of weakly chordal and bipartite graphs. We give a linear time algorithm to find an induced matching for such graphs and we prove that this algorithm indeed gives a maximum induced matching. Applying this algorithm, we provide a strongly biconvex graph whose (monomial) edge ideal does not admit a unique extremal Betti number. Using this constructed graph, we provide an infinite family of the so-called closed graphs (also known as proper interval graphs) whose binomial edge ideals do not have a unique extremal Betti number. This, in particular, answers the aforementioned question in \cite{AH}.   
	\end{abstract}

	
	
	\maketitle
	
	\section{Introduction}\label{introduction}
	
	Matchings are important and well-studied classical objects in graph theory. A certain type of matchings which provide an induced subgraph of the underlying graph, called an \emph{induced matching}, is also of interest in the literature. The maximum size of a matching in a graph $G$, denoted by $\nu(G)$, is called the \emph{matching number} of $G$, and the maximum size of an induced matching in $G$, denoted by $\inm(G)$ is called the \emph{induced matching number} of $G$. Induced matchings of graphs have many applications in the real world problems. They can be used to model uninterrupted communications between broadcasters and receivers. Induced matchings can also be used to capture a number of network problems, like network scheduling, gathering and testing. See for example \cite{BBKMT, BKMS, EGMT,GL}. 
	
	There have been and still are many attempts to find algorithms for  maximum (induced) matchings in the last decades. In \cite{SY}, a linear time algorithm was given for maximum matching in convex bipartite graphs, i.e. graphs whose bipartition admits a certain labeling. But, in general, finding a maximum induced matching in a graph is NP-hard, even in the class of bipartite graphs. Algorithms for finding a maximum induced matching were investigated in various families of graphs. In the case of bipartite graphs and biconvex graphs as a subclass of them were studied in \cite{DDL} and \cite{AS} respectively. In \cite{CST}, a polynomial time algorithm for finding a maximal induced matching in weakly chordal graphs was given while a linear time algorithm was provided for chordal graphs in \cite{BH}. In this paper, we give a linear time algorithm to find a maximum induced matching for a subclass of biconvex graphs which we call them \emph{strongly biconvex} graphs. It is observed that strongly biconvex graphs are also weakly chordal.  

    Maximum (induced) matchings also play role in the connection of graph theory and algebra. Recall that the (monomial) \emph{edge ideal} $I(G)$ of an $n$-vertex graph $G$ is the ideal in the polynomial ring $R=\KK[x_1,\ldots,x_n]$ generated by quadratics $x_ix_j$ where $\{i,j\}$ is an edge of $G$. The values $\inm(G)$ and $\nu(G)$ are lower and upper bounds for the (Castelnuovo-Mumford) regularity of the (monomial) edge ideal of a graph $G$, see \cite{Ka} and \cite{HV} respectively. In certain families of graphs, it is known that the lower bound is attained. Among them are weakly chordal graphs, see \cite{W}. 
    
    An algebraic topic of study in the case of (monomial) edge ideals which has been of interest of several authors, is the study of extremal Betti numbers of those ideals, see for example \cite{HKM}. A nonzero graded Betti number $\beta_{i,j}(R/I(G))$ is extremal if $\beta_{k,\ell}(R/I(G))=0$ for all $k\geq i$ and $\ell\geq j$ with $(k,\ell)\neq (i,j)$. A problem here is concerning uniqueness or non-uniqueness of the extremal Betti numbers. In this paper, benefiting from our algorithm, we construct a strongly biconvex graph whose (monomial) edge ideal does not have a unique extremal Betti number which is helpful for our further issues.   
    
    The same problem concerning the extremal Betti numbers has been also considered recently for another class of ideals attached to graphs, called \emph{binomial edge ideals}. The binomial edge ideal of a graph $G$, denoted by $J_G$, is the ideal in $S=\KK[x_1,\ldots,x_n,y_1,\ldots,y_n]$  generated by the binomials $x_iy_j-x_jy_i$. See \cite{HHHKR} and \cite{O}. The extremal Betti numbers of the binomial edge ideal of certain graphs were studied in \cite{AH} and \cite{HR}. In \cite{AH}, the authors also posed a question, see \cite[Question~1]{AH}. Indeed, the authors ask in this question if the initial ideals (with respect to the lexicographic order induced by $x_1>\cdots >x_n>y_1>\cdots >y_n$) of the so-called closed graphs have the unique extremal Betti number. Here we give a negative answer to this question which was in fact the first motivation of this paper.  
	
	The organization of this paper is as follows. In Section~\ref{biconvex}, we introduce strongly biconvex graphs as a subclass of biconvex graphs and, beside studying some of their properties, we provide our algorithm, which runs in linear time, for finding an induced matching for such graphs. We also show that the induced matching given by this algorithm is maximum. In Section~\ref{3-disjoint set of complete bipartite subgraphs}, we first recall the notion of strongly disjoint families of complete bipartite subgraphs from \cite{K} which is a key concept in the sequel for us. Then, we investigate the strongly disjoint families of complete bipartite subgraphs for strongly biconvex graphs and prove some lemmata which enable us to simplify the problems in the next section. Finally, Section~\ref{edge ideals} is devoted to the applications to the monomial and binomial edge ideals of graphs, respectively. As a consequence of some investigations of Section~\ref{3-disjoint set of complete bipartite subgraphs}, we give a formula for the projective dimension of the (monomial) edge ideals of strongly biconvex graphs in terms of certain subgraphs of them. We also construct a strongly biconvex graph $H_0$ such that $R/I(H_0)$ has more than one extremal Betti numbers. We prove this, by showing that $\beta_{p,p+4}(R/I(H_0))=0$, where $p=\projdim (R/I(H_0))$. To do this, a crucial tool is Kimura's non-vanishing theorem from \cite{K} as well as the fact that our algorithm indeed computes the regularity of $R/I(H_0)$. 
	Eventually, this graph leads us to provide an infinite family of closed graphs whose binomial edge ideals have more than one extremal Betti numbers which gives an affirmative answer to \cite[Question~1]{AH}.

	\section{Strongly biconvex graphs and their maximum induced matchings}\label{biconvex}
	
	In this section, we introduce a class of bipartite graphs, called strongly biconvex graphs, and investigate some of their properties. We also provide an algorithm to find an induced matching for strongly biconvex graphs and we show that this algorithm gives a maximum induced matching. We also show that this algorithm runs in linear time. 
	
	First, we recall the definition of convex bipartite graphs. Assume that $H$ is a bipartite graph with bipartition $X\cup Y$. For simplicity, we denote such a bipartite graph by $H=(X,Y)$. Let $E(H)$ be the edge set of $H$. Then $H$ is called $X$-\emph{convex} if there is an ordering on $X$ such that if $\{x_j,y_i\}\in E(H)$ and $\{x_k,y_i\}\in E(H)$ with $x_j, x_k\in X$ and $j < k$, then $\{x_p,y_i\}\in E(H)$ for all $p=j,\ldots,k$, (see for example \cite{SY}). A $Y$-\emph{convex} graph is defined similarly. 
	
	Recall that for any vertex $v$ of a graph $H$, the set of those  vertices of $H$ which are adjacent to $v$ is denoted by $N_{H}(v)$. The degree of $v$ in $H$, denoted by $\deg_H(v)$, is the number of elements of $N_{H}(v)$. It is easily seen that a bipartite graph $H=(X,Y)$ is $X$-convex (resp. $Y$-convex) if and only if $X$ (resp. $Y$) can be ordered so that the neighborhood of every vertex in $Y$ (resp. $X$) is labeled by a closed interval. Here, by a closed interval $[i,j]$ for $i<j$, we mean $\{s: i\leq s\leq j\}$. Half-closed intervals are defined accordingly. 
	
	A bipartite graph $H$ is called \emph{biconvex} if it is both $X$-convex and $Y$-convex (see for example \cite{AS}). Next, we introduce the new notion of strongly biconvex graphs which play an important role in this paper.
	
	\begin{Definition}\label{strongly biconvex}
		\em Let $H=(X,Y)$ be a bipartite graph with $X=\{x_q,x_{q+1},\dots,x_f\}$ and $Y=\{y_{q'},y_{q'+1},\ldots,y_{g}\}$ for some $q,q'\geq 1$. Then we call $H$ a \emph{strongly biconvex} graph (with respect to the given labeling) if the following conditions hold: 
		\begin{enumerate}
		 \item if $\{x_i,y_j\}\in E(H)$, then $i<j$; 
		 \item for any $r$ with $i<r<j$ and $\{x_i,y_j\}\in E(H)$, we have:
		   \begin{itemize}
		   	\item [(i)] if $x_r\in X$, then $\{x_r,y_j\}\in E(H)$; 
		   	\item [(ii)] if $y_r\in Y$, then $\{x_i,y_r\}\in E(H)$. 
		   \end{itemize} 
		\end{enumerate}   
	\end{Definition} 
	
	Note that in the last two conditions of the above definition, $x_r\in X$ or $y_r\in Y$ does not occur necessarily. Indeed, if $r>f$ or $r<q'$, then $x_r\notin X$ or $y_r\notin Y$, respectively. 
	
	The above definition is clearly based on a given labeling. We say that a graph is strongly biconvex if there exists a labeling for which the conditions of the above definition are fulfilled. Throughout the paper, when we say that $H=(X,Y)$ is a strongly biconvex graph, we mean with respect to the given labeling on $X$ and $Y$ as in Definition~\ref{strongly biconvex}. Note that by our definition, it is clear that any strongly biconvex graph is a biconvex graph. Figure~\ref{strongly biconvex example} depicts a strongly biconvex graph.   

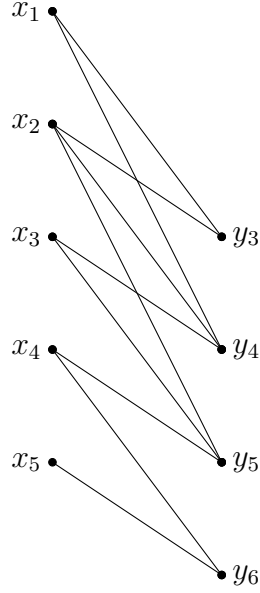
\begin{figure}[h!]
	\centering
	{	\begin{tikzpicture}[scale = 1.5]
		\begin{scope}
		\draw (0,4) node[punkt] {} -- (1.5,2) node[punkt] {};
		\draw (0,4) node[punkt] {} -- (1.5,1) node[punkt] {};
		\draw (0,3) node[punkt] {} -- (1.5,2) node[punkt] {};
		\draw (0,3) node[punkt] {} -- (1.5,1) node[punkt] {};
		\draw (0,3) node[punkt] {} -- (1.5,0) node[punkt] {};
		\draw (0,2) node[punkt] {} -- (1.5,1) node[punkt] {};
		\draw (0,2) node[punkt] {} -- (1.5,0) node[punkt] {};
		\draw (0,1) node[punkt] {} -- (1.5,0) node[punkt] {};
	    \draw (0,1) node[punkt] {} -- (1.5,-1) node[punkt] {};
		\draw (0,0) node[punkt] {} -- (1.5,-1) node[punkt] {};
		\node [left] at (0,0) {$x_5$};
		\node [left] at (0,1) {$x_4$};
		\node [left] at (0,2) {$x_3$};
		\node [left] at (0,3) {$x_2$};
		\node [left] at (0,4) {$x_1$};
		\node [right] at (1.5,-1) {$y_6$};
		\node [right] at (1.5,0) {$y_5$};
		\node [right] at (1.5,1) {$y_4$};
		\node [right] at (1.5,2) {$y_3$};
		\end{scope}
		\end{tikzpicture}
		\caption{A strongly biconvex graph}
		\label{strongly biconvex example}
	}
\end{figure}

\begin{Remark}\label{no isolated}
	{\em Let $H=(X,Y)$ be a strongly biconvex graph which does not have any isolated vertices. Then we have $q<q'$ and $f<g$, by condition~(1) in Definition~\ref{strongly biconvex}, and moreover condition~(2) of the definition implies that $\{x_q,y_{q'}\}$ and $\{x_f,y_g\}$ are both edges of $H$. } 
\end{Remark}	
	
	For a strongly biconvex graph $H=(X,Y)$, we set 
	\[
	m(i)=\max \{q',i+1\}
	\]
	and 
	\[
	M(i)=\max \{t: \{x_i,y_t\}\in E(H)\},
	\]
	for any $i=q,\ldots, f$ where $x_i$ is not an isolated vertex of $H$.   
	
	In the next proposition, an equivalent condition for being a strongly biconvex graph is given.

	\begin{Proposition}\label{equivalent to strongly biconvex}
		Let $H=(X,Y)$ be a bipartite graph with $X=\{x_q,x_{q+1},\dots,x_f\}$ and $Y=\{y_{q'},y_{q'+1},\ldots,y_{g}\}$ which has no isolated vertices. Then $H$ is strongly biconvex if and only if the following conditions hold: 
	\begin{enumerate}
		\item [\em(a)] $N_H(x_i)=\{y_t: t\in [m(i),M(i)]\}$ for any $i=q,\ldots,f$;
		\item [\em(b)] $M(i)\leq M(j)$ for any $i,j$ with $q\leq i<j\leq f$. 
	\end{enumerate}
	\end{Proposition}

\begin{proof}
	Suppose that $H$ is a strongly biconvex graph. First we prove~(a). Let $x_i\in X$. We show that $\{x_i,y_{m(i)}\}\in E(H)$. If $q'\geq i+1$, then $\{x_i,y_{q'=m(i)}\}\in E(H)$, since $q\leq i<q'$ and since by Remark~\ref{no isolated} we have $\{x_q,y_{q'}\}\in E(H)$. If $q'< i+1$, then clearly $m(i)=i+1$. Since $x_i$ is not an isolated vertex, there exists some $j$ with $i<i+1\leq j$ such that $\{x_i,y_j\}\in E(H)$, and hence $\{x_i,y_{i+1=m(i)}\}\in E(H)$. On the other hand, by definition of $M(i)$, it is clear that $\{x_i,y_{M(i)}\}\in E(H)$. Now, let $m(i)< r <M(i)$. 
	Thus, it follows from $\{x_i,y_{M(i)}\}\in E(H)$ that $\{x_i,y_r\}\in E(H)$. Therefore, by definitions of $m(i)$ and $M(i)$ part~(a) follows. 
	
	Next we prove (b). Let $i,j\in \{q,\ldots, f\}$ with $i<j$. If $j\geq M(i)$, then the desired inequality in~(b) holds, since clearly we have $j<M(j)$. Now assume that $j<M(i)$. Since $i<j$ and $\{x_i,y_{M(i)}\}\in E(H)$, it follows that $\{x_j,y_{M(i)}\}\in E(H)$. Hence, $M(i)\leq M(j)$ by the definition of $M(j)$, as desired. 
	
	\medskip
	Conversely, suppose that the conditions (a) and (b) hold for $H$. We show that $H$ is strongly biconvex. Assume that $\{x_i,y_j\}\in E(H)$ for some $x_i\in X$ and $y_j\in Y$. Thus, $y_j\in N_H(x_i)$, and hence by~(a) we have $j\geq m(i)$. This together with the fact that $i<m(i)$ imply that $i<j$ which fulfills condition~(1) in Definition~\ref{strongly biconvex}. 
	
	Next, let $x_i\in X$ and $y_j\in Y$ be such that $\{x_i,y_j\}\in E(H)$ and let $i<r<j$. Assume that $x_r\in X$. We show that $\{x_r,y_j\} \in E(H)$. Since $j\geq r+1$, we have $j\geq m(r)$. On the other hand, $j\leq M(i)$, because $\{x_i,y_j\}\in E(H)$. Since $i<r$, by condition~(b) we get $M(i)\leq M(r)$, and hence $j\leq M(r)$. Therefore, by condition~(a) it follows that $\{x_r,y_j\} \in E(H)$.   
	
	Assume $y_r\in Y$. We show that $\{x_i,y_r\} \in E(H)$. It follows from $\{x_i,y_j\}\in E(H)$ that $j\leq M(i)$, and hence $r<M(i)$. 
	Since $r>i$, we have $r\geq m(i)$. Therefore, $m(i)\leq r < M(i)$, and hence by condition~(a) we deduce that $\{x_i,y_r\} \in E(H)$. So, condition~(2) in Definition~\ref{strongly biconvex} is also satisfied, and hence $H$ is strongly biconvex. 
\end{proof}

	
Recall that a graph $H$ is called \emph{weakly chordal} if neither $H$ nor its complementary graph $H^c$ has an induced cycle of length greater than $4$. 
It is known that any biconvex graph is weakly chordal. In the following, for the convenience of the reader we give a proof in the case of strongly biconvex graphs. 

\begin{Proposition}\label{weakly chordal}
 Any strongly biconvex graph is weakly chordal. 
\end{Proposition}

\begin{proof}
	Let $H=(X,Y)$ be a strongly biconvex graph, and let $C$ be an induced cycle in $H$ labeled as $x_{\alpha_1},y_{\beta_1},x_{\alpha_2},y_{\beta_2},\ldots, x_{\alpha_t},y_{\beta_t},x_{\alpha_1}$ with $t\geq 3$. We may assume that $\alpha_1<\alpha_i$ for all $i=2,\ldots,t$. If $\alpha_2<\beta_t$, then we get $\{x_{\alpha_2},y_{\beta_t}\}\in E(H)$, since $\{x_{\alpha_1},y_{\beta_t}\}\in E(H)$. This is a contradiction to the fact that $C$ is an induced cycle. So assume that $\beta_t \leq \alpha_2$. Thus, we have $\alpha_1<\alpha_t<\beta_t\leq \alpha_2<\beta_1$, where the second and the last inequalities follow because $\{x_{\alpha_t},y_{\beta_t}\}\in E(H)$ and $\{x_{\alpha_2},y_{\beta_1}\}\in E(H)$. Since $\{x_{\alpha_1},y_{\beta_1}\}$ is an edge of $H$, it follows that $\{x_{\alpha_t},y_{\beta_1}\}$ is an edge too, a contradiction to the fact that $C$ is an induced cycle. Therefore, $H$ does not have any induced cycle of length greater than~$4$. On the other hand, since $H$ is bipartite, it is clear that any induced cycle in $H^c$ has length at most~$4$. Thus, $H$ is a weakly chordal graph, as desired.  
\end{proof}

Finding a maximum matching as well as a maximum induced matching in bipartite graphs and, in particular, in  convex bipartite graphs has been an interesting problem considered by several authors, see for example \cite{DDL, SY}.

In the following theorem indeed we provide an algorithm to find a maximum induced matching for any strongly biconvex graph. This algorithm is of \emph{greedy} type. Recall that an induced matching in a graph is a set of disjoint edges whose endpoints are not adjacent to each other. Such edges are also called \emph{pairwise $3$-disjoint}. A maximum induced matching in a graph is an induced matching of the  maximum size. The size of a maximum induced matching in $H$ is called the \emph{induced matching number} and is denoted by $\inm(H)$. 

Before stating the next theorem, we fix some notation. Let $H=(X,Y)$ be a strongly biconvex graph with no isolated vertices, and let $i_1=q$ and $j_1=q'$. For any $\ell\geq 2$, we set 
\[
T^{\ell}_X=\{t: t\geq j_{\ell-1}~,~N_H(x_t)\not \subseteq N_H(x_{i_{\ell-1}}) \}. 
\] 
If $T^{\ell}_X\neq \emptyset$, then we set 
\[
i_{\ell}=\min T^{\ell}_X,
\]
\\
\[
T^{\ell}_Y=\{t: y_t\in N_H(x_{i_{\ell}})\setminus N_H(x_{i_{\ell-1}}) \}
\]  
and 
\[
j_{\ell}=\min T^{\ell}_Y.
\]
Now let $m$ be the biggest integer for which $T^{m}_X\neq \emptyset$. Then consider the following set of edges of $H$:
\[
\mathcal{M}(H)=\big \{ \{x_{i_\ell},y_{j_{\ell}}\}: \ell=1,\ldots , m \big \}.
\] 

Using the above notation, we have the following: 

\begin{Theorem}\label{induced matching}
	Let $H=(X,Y)$ be a strongly biconvex graph with no isolated vertices. Then $\MM(H)$	is a maximum induced matching for $H$.   
\end{Theorem} 	

\begin{proof}
Let $\MM=\MM(H)$. Note that by Definition~\ref{strongly biconvex} and the choice of $i_{\ell}$, we have 
\begin{equation}\label{label ordering}
i_{\ell-1}<j_{\ell-1} \leq i_{\ell}<j_{\ell}
\end{equation} 	
for any $\ell=2,\ldots,m$. 

First we show that $\mathcal{M}$ is an induced matching of $H$. Let $\ell=2,\ldots,m$. Then by the choice of $i_{\ell}$ and $j_{\ell}$, it is clear that $\{x_{i_{\ell}},y_{j_{\ell-1}}\}\notin E(H)$ and $\{x_{i_{\ell-1}},y_{j_{\ell}}\}\notin E(H)$. Now, let $t<\ell-1$. By the structure of $H$, it is clear that $\{x_{i_{\ell}},y_{j_t}\}\notin E(H)$, since $\ell >t$. If $\{x_{i_t},y_{j_{\ell}}\}\in E(H)$, then by definition of a strongly biconvex graph, it follows that $\{x_{i_{\ell-1}},y_{j_{\ell}}\}\in E(H)$, a contradiction. Therefore $\mathcal{M}$ is an induced matching of size $m$ for $H$. 

Next we show that $\mathcal{M}$ is a maximum induced matching for $H$. For this, suppose that 
\[
\mathcal{M}'=\big \{ \{x_{{\alpha}_i},y_{{\beta}_i}\}: i=1,\ldots,r \big\}
\]    
is an induced matching of size $r$ for $H$. Then, it is enough to show that $r\leq m$. We may assume that $\alpha_1<\cdots < \alpha_r$. For any $i=1,\ldots,r-1$, we have $\beta_i\leq \alpha_{i+1}$. Otherwise,  $\alpha_{i+1}< \beta_i$ together with $\alpha_{i}<\alpha_{i+1}$ implies that $\{x_{\alpha_{i+1}}, y_{\beta_i}\}\in H$, since $\{x_{\alpha_{i}}, y_{\beta_i}\}\in H$. This is a contradiction to the fact that $\mathcal{M}'$ is an induced matching. Therefore, for any $i=1,\ldots,r-1$, we have
\begin{equation}\label{ordering of alpha-beta}
\alpha_i < \beta_i \leq \alpha_{i+1}.
\end{equation}
  
Let $I_{\ell}=[i_{\ell},i_{\ell+1})$ for $\ell=1,\ldots,m-1$, and let $I_m=[i_m,f]$. If $m=1$, then we only have one interval $I_1=[q,f]$. In this case we show that $\inm(H)=1$, and hence $r=m=1$. First note that by the structure of $H$ we have $\{x_f,y_g\}\in E(H)$, since $H$ does not have any isolated vertices. Now, we distinguish two cases: 

(i)~Suppose that $f<q'$. Then $N_H(x_f)=\{y_s: s\in [q',g]\}$, since $\{x_f,y_g\}\in E(H)$ and $f<q'\leq s\leq g$. On the other hand, $N_H(x_t)=\{y_s: s\in [q',M(t)]\}$ for all $t=q,\ldots, f-1$. Thus, we have $N_H(x_q)\subseteq N_H(x_{q+1})\subseteq \cdots \subseteq N_H(x_f)$ which implies that there are no two $3$-disjoint edges in $H$, and hence $\inm(H)=1$.  

(ii)~Suppose that $f\geq q'$. If $\{x_q,y_g\}\notin E(H)$, then $y_g\in N_H(x_f)\setminus N_H(x_q)$. So, $T^{2}_X\neq \emptyset$ and hence $m\geq 2$, a contradiction. Thus, $\{x_q,y_g\}\in E(H)$, and hence $N_H(x_q)=\{y_s: s\in [q',g]\}$. Since $\{x_q,y_g\}\in E(H)$, we have $\{x_t,y_g\}\in E(H)$ for any $t$ with $q<t<f<g$. Therefore, $N_H(x_t)=\{y_s: s\in [m(t),g]\}$ where for $t<q'$, $m(t)=q'$ while for $t\geq q'$, $m(t)=t+1>q'$. This implies that $N_H(x_f)\subseteq \cdots \subseteq N_H(x_{q+1})\subseteq N_H(x_q)$, and hence there do not exist any two $3$-disjoint edges in $H$, namely $\inm(H)=1$.  

Now assume that $m\geq 2$. Suppose that $I_{\ell}$, for some $\ell=1,\ldots,m-1$, contains at least two of $\alpha_i$'s, say $\alpha_t$ and $\alpha_{t+1}$. In the following, we show that $j_{\ell}=i_{\ell+1}$. 

Note that we have 
\[
i_{\ell}\leq \alpha_t<\beta_t\leq \alpha_{t+1}<\beta_{t+1}.
\]
So, if $\{x_{i_{\ell}},y_{\beta_{t+1}}\}\in E(H)$, then $\{x_{\alpha_t},y_{\beta_{t+1}}\}\in E(H)$, a contradiction, since $\MM'$ is an induced matching. Therefore, 
\begin{equation}\label{beta t+1}
\{x_{i_{\ell}},y_{\beta_{t+1}}\}\notin E(H).
\end{equation}
Thus, it follows that 
\begin{equation}\label{alpha t+1}
\alpha_{t+1}<j_{\ell}
\end{equation}
by the choice of $i_{\ell+1}$, since $\alpha_{t+1}<i_{\ell+1}$. 

If $\beta_{t+1}<j_{\ell}$, then we have 
$\{x_{i_{\ell}},y_{\beta_{t+1}}\}\in E(H)$, because 
$\beta_{t+1}\geq q'$, $i_{\ell}<\beta_{t+1}$ and $\{x_{i_{\ell}},y_{j_{\ell}}\}\in E(H)$. But this is a contradiction to \eqref{beta t+1}, and hence we have $j_{\ell}<\beta_{t+1}$, since clearly $j_{\ell}\neq \beta_{t+1}$. The latter inequality together with \eqref{alpha t+1} implies that 
\begin{equation}\label{j ell and beta t+1}
\{x_{j_{\ell}},y_{\beta_{t+1}}\}\in E(H),
\end{equation}
since $\{x_{\alpha_{t+1}},y_{\beta_{t+1}}\}\in E(H)$. 
By the choice of $i_{\ell+1}$ and by \eqref{beta t+1} and \eqref{j ell and beta t+1}, we get $j_{\ell}\geq i_{\ell+1}$. So, \eqref{label ordering} implies that $j_{\ell}=i_{\ell+1}$, as desired. In particular, it follows that $\ell\geq 2$. Indeed, if $\ell=1$, then we have $i_2=j_1=q'$, and hence $q\leq \alpha_1<\beta_1 \leq \alpha_2< q'$, a contradiction, since $q'$ is the smallest index for the elements of $Y$.  

Note that if $\alpha_r\geq j_m$, then $f\geq j_m$ and $\alpha_r\in I_m$. In this case, we show that $\alpha_t\notin I_m$ for any $t<r$. By our ordering, it is enough to show that $\alpha_{r-1}\notin I_m$. Suppose on contrary that $\alpha_{r-1}\in I_m$. Then we have 
\begin{equation}\label{last interval}
i_m\leq \alpha_{r-1}<\beta_{r-1}\leq \alpha_r<\beta_r,
\end{equation}
by \eqref{ordering of alpha-beta}. If $\{x_{i_m},y_{\beta_r}\}\notin E(H)$, then by definition of $\MM$, one could add $\{x_{\alpha_r},y_s\}$, for some $s\leq \beta_r$, to $\MM$, a contradiction. So, $\{x_{i_m},y_{\beta_r}\}\in E(H)$, which implies together with \eqref{last interval} that $\{x_{\alpha_{r-1}},y_{\beta_r}\}\in E(H)$. The latter is a contradiction to the fact that $\MM'$ is an induced matching, and hence we have $\alpha_{r-1}\notin I_m$. 

Next we show that none of $I_1,\ldots,I_m$ can contain three of $\alpha_i$'s. Assume that $\alpha_{t-1},\alpha_t,\alpha_{t+1}\in I_{\ell}$ for some $t=2,\ldots,r-1$ and $\ell=1,\ldots,m$. In the particular case of $\ell=m$, we have $t=r-1$ and $\alpha_r<j_m$. This combined with \eqref{ordering of alpha-beta} and \eqref{alpha t+1} implies that  
\[
i_{\ell}\leq \alpha_{t-1}<\beta_{t-1}\leq \alpha_t<\beta_t\leq \alpha_{t+1}<j_{\ell}.
\]      
Since $\{ x_{i_{\ell}},y_{j_{\ell}} \}\in E(H)$, it follows that 
$\{ x_{\alpha_{t-1}},y_{j_{\ell}} \}\in E(H)$,
and hence $\{x_{\alpha_{t-1}},y_{\beta_t}\}\in E(H)$, a contradiction.

Therefore, we have already shown that $I_1$ contains at most one of $\alpha_i$'s and any of $I_2,\ldots,I_m$ contains at most two of $\alpha_i$'s. Finally, we show that if $I_{\ell}$ contains two of $\alpha_i$'s for some $\ell=2,\ldots,m$, then $I_{\ell-1}$ contains none of them. This then shows that $r\leq m$ and completes the proof. Let $\alpha_t,\alpha_{t+1}\in I_{\ell}$. If $\alpha_{t-1}\in I_{\ell-1}$, then by \eqref{ordering of alpha-beta} we have 
\begin{equation}\label{finally}
i_{\ell-1}\leq \alpha_{t-1}<\beta_{t-1}\leq \alpha_t<\beta_t.
\end{equation}   
On the other hand, by \eqref{ordering of alpha-beta} and \eqref{alpha t+1}, we have 
\[
i_{\ell}\leq \alpha_t<\beta_t\leq \alpha_{t+1}<j_{\ell},
\]
(here, $t+1$ could be also $r$ by our assumptions on $I_{\ell}$). 
Thus  $\{x_{i_{\ell}},y_{\beta_t}\}\in E(H)$, since $\{x_{i_{\ell}},y_{j_{\ell}}\}\in E(H)$. As $\beta_t<j_{\ell}$, it follows from the choice of $j_{\ell}$ that $\{x_{i_{\ell-1}},y_{\beta_t}\}\in E(H)$. Combining this with \eqref{finally}, we get $\{x_{\alpha_{t-1}},y_{\beta_t}\}\in E(H)$ which is a contradiction, since $\MM'$ is an induced matching for $H$. Therefore, $\alpha_{t-1}\notin I_{\ell-1}$. Our ordering on $\alpha_i$'s, yields that none of $\alpha_i$'s belongs to $I_{\ell-1}$, as desired.        
\end{proof}

\begin{Remark}\label{linear time}
{\em According to the notation of Theorem~\ref{induced matching}, we would like to remark that one could observe that 
	\begin{equation}\label{j_ell}
    j_{\ell}=M(i_{\ell-1})+1
	\end{equation} 
for any $\ell=2,\ldots,m$. Indeed, by the choice of $j_{\ell}$ and Proposition~\ref{equivalent to strongly biconvex}, we have $j_{\ell}\notin [m(i_{\ell-1}),M(i_{\ell-1})]$. This implies that $j_{\ell}\geq M(i_{\ell-1})+1$, since clearly we have $j_{\ell}\geq m(i_{\ell})\geq m(i_{\ell-1})$. On the other hand, by the choice of $i_{\ell}$, it follows that $i_{\ell}\leq M(i_{\ell-1})$, since none of the neighbors of $x_{M(i_{\ell-1})}$ is adjacent to $x_{i_{\ell-1}}$.  So, we have $i_{\ell} < M(i_{\ell-1})+1\leq M(i_{\ell})$, where the last inequality follows from the fact that $x_{i_{\ell}}$ has a neighbor which is not a neighbor of $x_{i_{\ell-1}}$. Therefore, $\{x_{i_{\ell}},y_{M(i_{\ell-1})+1}\}\in E(H)$. Then it follows that $j_{\ell}\leq M(i_{\ell-1})+1$, because $y_{M(i_{\ell-1})+1}$ is clearly not adjacent to $x_{i_{\ell-1}}$. }
\end{Remark}

Given a labeled strongly biconvex graph $H=(X,Y)$ and having $M(i)$'s for all $i$, the observation \eqref{j_ell} in Remark~\ref{linear time} implies that a maximum induced matching in $H$ can be found in a linear time, namely $O(|X|)$. So, we have the following corollary:
  

\begin{Corollary}\label{o(n)}
	A maximum induced matching in a (labeled) strongly biconvex graph can be computed in a linear time.  
\end{Corollary}

	

    \section{Strongly disjoint families of complete bipartite subgraphs in strongly biconvex graphs}\label{3-disjoint set of complete bipartite subgraphs}
     
     In this section, we investigate about the properties of strongly disjoint families of complete bipartite subgraphs (in the sense of \cite{K}) of a strongly biconvex graph. The results of this section enables us to give an affirmative answer to \cite[Question~1]{AH} in the next section. 
     
     First we recall some definitions and fix some notation. Let $G$ be a graph. The family $\MB=\{B_1,\ldots,B_r\}$ of complete bipartite subgraphs of $G$ is called \emph{strongly disjoint} if the following conditions hold:
     \begin{enumerate}
     	\item $V(B_k)\cap V(B_{\ell})= \emptyset$ for all $k\neq \ell$;
     	\item for each $i=1,\ldots,r$, there exists $e_i\in E(B_i)$ such that $\{e_1,\ldots,e_r\}$ is an induced matching for $G$. 
     \end{enumerate} 
 
 Given a strongly disjoint family $\MB$ of complete bipartite subgraphs of $G$, we set 
 \[
 V(\MB)=\cup_{i=1}^r V(B_i)
 \] 
 and
 \[
 d(\MB)=\sum_{i=1}^{r}|V(B_i)| - r.
 \]
 We also set $\MS(G)$ to be the set of all strongly disjoint families of complete bipartite subgraphs of $G$, and  
\[
d(G)=\max \{d(\MB): \MB \in \MS(G)\}.
\]  

Now, let $H=(X,Y)$ be a strongly biconvex graph with no isolated vertex. For any $\MB=\{B_1,\ldots,B_r\}\in \MS(H)$, we set $X(B_i)=V(B_i)\cap X$ and $Y(B_i)=V(B_i)\cap Y$ for any $i=1,\ldots,r$. We also let $m(B_i)$ and $M(B_i)$ be the minimum and the maximum index of a vertex in $X(B_i)$ for any $i$, respectively. Also, we set $m'(B_i)$ and $M'(B_i)$ to be the minimum and the maximum index of a vertex in $Y(B_i)$ for any $i$, respectively.

For any subset $T$ of the vertices of a graph $G$, we denote the induced subgraph of $G$ on $V(G)\setminus T$ by $G-T$. In particular, if $T$ consists of only one vertex $v$, then we simply write $G-v$.


\begin{Lemma}\label{ordered families}
	Let $H=(X,Y)$ be a strongly biconvex graph with no isolated vertex, and let $\MB=\{B_1,\ldots,B_r\}\in \MS(H)$. Then there exists $\MB'=\{B'_1,\ldots,B'_r\}\in \MS(H)$ with the following properties: 
	\begin{enumerate}
		\item [\em(a)] $M(B'_i)<m(B'_j)$ and $M'(B'_i)<m'(B'_j)$ for any $i<j$;
		\item [\em(b)] $X(B'_i)$ and $Y(B'_i)$ are indexed by some intervals for all $i=1,\ldots,r$;
		\item [\em(c)] $d(\MB)\leq d(\MB')$.
	\end{enumerate} 
\end{Lemma}

\begin{proof}
 We may assume that the vertex with minimum index among the vertices of $X(B_i)$'s is $x_{m(B_1)}$. Let $m(B_1)<k< M(B_1)$. Then $x_k$ is adjacent to all the vertices in $Y(B_1)$, since $H$ is strongly biconvex. We add all such $x_k$'s to $X(B_1)$ and obtain a subset of $X$ which is clearly indexed by the interval $[m(B_1),M(B_1)]$ and we denote it by $X'_1$. Note that $x_k$'s might be among the vertices of $B_i$'s or not. Similarly, we can add all $y_k$'s with $k\in [m'(B_1),M'(B_1)]$ to $Y(B_1)$ to obtain a subset $Y'_1$ of vertices which is indexed by an interval. Therefore, we gain a desired complete bipartite subgraph $B'_1$ of $H$ with $X(B'_1)=X'_1$ and $Y(B'_1)=Y'_1$. Note that $y_{m'(B_1)}$ has the minimum index among the vertices of $Y(B'_1),Y(B_2),\ldots, Y(B_r)$. Indeed, if $y_j\in Y(B_{\ell})$ for some $j<m'(B_1)$ and $\ell>1$, then any vertex $x_i$ from $X(B_{\ell})$, which are now all indexed bigger than $M(B_1)$, is adjacent to $y_k$ for all $k\in [m'(B_1),M'(B_1)]$. This is then a contradiction, because of the existence an induced matching of size $r$.      
  We denote the remaining subgraphs of the complete bipartite graphs $B_2,\ldots,B_r$, by $\tilde{B_2},\ldots, \tilde{B_r}$. The graph 
 $H'=H-\{x_i,y_j:i\leq M(B_1), j\leq M'(B_1)\}$ is obviously a strongly biconvex graph. Then, it follows that $\tilde{\MB}=\{\tilde{B_2},\ldots,\tilde{B_r}\}\in \MS (H')$. Note that by the above procedure, we still remain with exactly $r$ complete bipartite graphs, since $\MB$ admits an induced matching of size $r$. Therefore, we have 
 \begin{equation}\label{d(B)}
 d(\MB)\leq d(\tilde{\MB}\cup \{B'_1\}).
 \end{equation} 
 Finally, induction on $r$ implies that there exists $\{B'_2,\ldots,B'_r\}\in \MS(H')$ with conditions (a), (b) and (c) in comparison with $\tilde{\MB}$. We let $\MB'=\{B'_1,B'_2,\ldots,B'_r\}$ which clearly belongs to $\MS(H)$. By our procedure, it is also clear that $M(B'_1)<m(B'_j)$ and $M'(B'_1)<m'(B'_j)$ for any $j>1$. Moreover, we have $d(\MB)\leq d(\MB')$ by \eqref{d(B)} and the induction hypothesis. Hence, $\MB'$ is  inductively constructed. 
\end{proof}

\begin{Lemma}\label{top matching}
  Let $H=(X,Y)$ be a strongly biconvex graph with no isolated vertex, and let $\MB=\{B_1,\ldots,B_r\} \in \MS(H)$ which satisfies conditions~{\em(}a{\em)} and~{\em(}b{\em)} in Lemma~\ref{ordered families}. Then there exists $\MB'=\{B'_1,\ldots,B'_r\} \in \MS(H)$  for which the set of edges 
  \[
  \MT=\big{\{} \{x_{m(B'_i)},y_{m'(B'_i)}\} : i=1,\ldots,r \big{\}}
  \] 
  is an induced matching of $H$ and $d(\MB)\leq d(\MB')$.
\end{Lemma}

\begin{proof}
  Let $X'_1$ be obtained by adding all the vertices $x_k\in X(B_2)$ which are adjacent to $y_{m'(B_1)}$ to the set $X(B_1)$. Also, let $Y'_1$ be obtained by adding all the vertices $y_{\ell}\in Y(B_2)$ which are adjacent to $x_{m(B_1)}$ to the set $Y(B_1)$. 
  Now we set $B'_1$ to be the complete bipartite subgraph of $H$ with $X(B'_1)=X'_1$ and $Y(B'_1)=Y'_1$. We also denote the remaining subgraph of $B'_2$, by $\tilde{B}_2$. Since $\MB$ admits an induced matching of size $r$ arising from each $B_i$, it follows that not all elements of $X(B_2)$ (resp. $Y(B_2)$) are moved into $X(B'_1)$ (resp. $Y(B'_1)$). By the construction of $H$, it is also obvious that none of the elements of $X(B_i)$ and $Y(B_i)$ for $i>2$, are adjacent to $y_{m'(B_1)}$ and $x_{m(B_1)}$, respectively. So, we obtain $\{B'_1,\tilde{B_2},B_3,\ldots,B_r\}\in \MS(H)$ such that clearly we have 
  \begin{equation}\label{d(tilde{B})}
  d(\MB)\leq d(\{B'_1,\tilde{B_2},B_3,\ldots,B_r\}).
  \end{equation}  
  The graph $H'=H-\{x_i,y_j:i\leq M(B'_1), j\leq M'(B'_1)\}$ is a strongly biconvex graph and $\{\tilde{B_2},B_3,\ldots,B_r\}\in \MS(H')$. Hence, by induction on $r$ it follows that there exists $\tilde{\MB}=\{B'_2,\ldots,B'_r\}\in \MS(H')$ with 
  $d(\{\tilde{B_2},B_3,\ldots,B_r\})\leq d(\tilde{\MB})$. We let $\MB'=\{B'_1\}\cup \tilde{B}$ which is in $\MS(H)$ and using \eqref{d(tilde{B})} we get $d(\MB)\leq d(\MB')$. The induction hypothesis also yields that the edges $\{x_{m(B'_i)},y_{m'(B'_i)}\}$ for $i=2,\ldots,r$ provide an induced matching for $H$. Our procedure to construct $\MB'$ implies that the edge $\{x_{m(B'_1)},y_{m'(B'_1)}\}$ could be also added to this induced matching, as desired.   
\end{proof}

Let $H$ be a strongly biconvex graph. Let $\MB\in \MS (H)$ which satisfies the conditions~(a) and~(b) of Lemma~\ref{ordered families} such that the set of edges $\MT$ of Lemma~\ref{top matching} provides an induced matching for it. Then, for simplicity, we call $\MB$ an \emph{ordered} strongly disjoint family of complete bipartite subgraphs of $H$. We denote by $\mathcal{OS}(H)$ the set of all such families for $H$.  

For any strongly biconvex graph $H=(X,Y)$, if $e=\{x_q,y_{q'}\}$ is an edge of $H$, then we denote the induced subgraph of $H$ on the set of vertices $N_H(x_q)\cup N_H(y_{q'})$ by $B_e$. It is easily seen that $B_e$ is a complete bipartite subgraph of $H$.

\begin{Theorem}\label{B_e}
	Let $H=(X,Y)$ be a strongly biconvex graph with no isolated vertex, and let $\MB=\{B_1,\ldots,B_r\}\in \mathcal{OS}(H)$ with $d(\MB)=d(H)$. If $x_q\in V(B_1)$ and $e=\{x_q,y_{q'}\}$, then  $\MB'=\{B_e,B_2,\ldots,B_r\}\in \mathcal{OS}(H)$ with $d(\MB')=d(H)$.   
\end{Theorem}

\begin{proof}
	We distinguish the following two cases. 
	
	Case~1. Suppose that $M(B_1)<q'$. If $q'<m'(B_1)$, then all vertices in $X(B_1)$ are adjacent to $y_{q'}$ in $H$. So, by adding $y_{q'}$ to $Y(B_1)$, one could replace $B_1$ in $\MB$ with a complete bipartite subgraph with one more vertex, which contradicts the assumption $d(\MB)=d(H)$. Therefore, we have $q'=m'(B_1)$. If $\{x_q,y_{M'(B_1)+1}\}\in E(H)$, then it follows that $y_{M'(B_1)+1}$ is not a vertex of $B_2$ and hence any of $B_i$'s in $\MB$. Otherwise, it participates in the induced matching $\MT$ of Lemma~\ref{top matching}, a contradiction. Thus, by adding $y_{M'(B_1)+1}$ to $Y(B_1)$, again we can replace $B_1$ with a complete bipartite graph with more vertices, contradicting $d(\MB)=d(H)$. This implies that $M'(B_1)$ is the maximum index that a neighbor of $x_q$ has, and hence 
	\[
	N_H(x_q)=\{y_j: j\in [q',M'(B_1)]\}.
	\]
	Similarly, if $\{x_{M(B_1)+1},y_{q'}\}\in E(H)$, then by adding $x_{M(B_1)+1}$ to $X(B_1)$, one gets a contradiction to $d(\MB)=d(H)$. Thus, $M(B_1)+1$ is the maximum index of the neighbors of $y_{q'}$, and hence 
	\[
	N_H(y_{q'})=\{x_i: i\in [q,M(B_1)]\}.
	\]    
	Therefore, in this case we have $B_1=B_e$. 
	
	Case~2. Suppose that $q'\leq M(B_1)$. Clearly, we have $m'(B_1)\geq M(B_1)+1$. If $m'(B_1) > M(B_1)+1$, then all the vertices of $X(B_1)$ are adjacent to $y_{M(B_1)+1}$. Therefore, similar to the previous case, we may add this vertex to $Y(B_1)$, which  contradicts the assumption $d(\MB)=d(H)$. Thus, we have $m'(B_1)=M(B_1)+1$. Now, let $B'_1$ be the complete bipartite subgraph of $H$ on the vertex set $\{x_q,\ldots,x_{q'-1}\}\cup \{y_{q'},\ldots,y_{M'(B_1)}\}$, and let $\MB'=\{B'_1,B_2,\ldots,B_r\}$. Then it is easily seen that $\MB'\in \mathcal{OS}(H)$. On the other hand, we have $|V(B'_1)|=|V(B_1)|$ which implies that $d(\MB')=d(\MB)$. Now, it is enough to verify that $B'_1=B_e$ which follows from the first case, since we have $M(B'_1)=q'-1<q'$.       
\end{proof}

\begin{Corollary}\label{d(H)}
  Let $H=(X,Y)$ be a strongly biconvex graph with no isolated vertices and let $e=\{x_q,y_{q'}\}$. Then 
  \[
  d(H)=\max \{d(H-V(B_e))+|V(B_e)|-1,d(H-x_q)\}. 
  \]   
\end{Corollary}

\begin{proof}
	Let $\MB=\{B_1,\ldots,B_r\}\in \mathcal{OS}(H)$ with $d(\MB)=d(H)$. If $x_q\notin V(B_1)$, then it is clear that $\MB\in \mathcal{OS}(H-x_q)$ which implies that $d(H)\leq d(H-x_q)$. If $x_q\in V(B_1)$, then by Theorem~\ref{B_e} we have $\MB'=\{B_e,B_2,\ldots,B_r\}\in \mathcal{OS}(H)$ and $d(\MB')=d(H)$. It is easily seen that $H-V(B_e)$ is a strongly biconvex graph and $\{B_2,\ldots,B_r\}\in \mathcal{OS}(H-V(B_e))$. This implies that $d(H)\leq d(H-V(B_e))+|V(B_e)|-1$. Thus, 
	$d(H)\leq \max \{d(H-V(B_e))+|V(B_e)|-1,d(H-x_q)\}$. On the other hand, by definitions, it easily follows that 
	$d(H)\geq \max \{d(H-V(B_e))+|V(B_e)|-1,d(H-x_q)\}$, since $H-x_q$ and $H-V(B_e)$ are induced subgraphs of $H$. Therefore, the desired equality holds.    
\end{proof}

\section{Extremal Betti numbers of monomial and binomial edge ideals of graphs}\label{edge ideals}
	
 In this section, we study the extremal Betti numbers of some monomial and binomial ideals associated to graphs. The main goal of this section is to provide certain strongly biconvex graphs whose monomial/binomial edge ideals do not have a unique extremal Betti number. This, in particular, provides a negative answer to \cite[Question~1]{AH}.   
	
 Let $R=\KK[x_1,\ldots,x_n]$ be a polynomial ring over a field $\KK$ and let $I$ be a homogeneous ideal in $R$. Also let  
	\[
	0 \rightarrow \bigoplus_{j} R(-j)^{\beta_{p,j}(R/I)} \rightarrow 
	\cdots \rightarrow \bigoplus_j R(-j)^{\beta_{1,j}(R/I)}\rightarrow 
	R\rightarrow R/I \rightarrow 0  
	\]
    be the \emph{minimal (standard) $\ZZ$-graded free resolution} of $R/I$ over $R$ with $\deg (x_i)=1$ for all $i$. Here $p$ is the \emph{projective dimension} of $R/I$, denoted by $\projdim (R/I)$, and $\beta_{i,j}(R/I)$ is the $(i,j)$-graded Betti number of $R/I$. The \emph{Castelnuovo-Mumford regularity} of $R/I$ is 
    \[
    \reg (R/I)=\max \{j-i: \beta_{i,j}(R/I)\neq 0\}.
    \]
    Considering the natural $\ZZ^n$-grading of $R$ given by $\deg (x_i)=e_i$, instead of the standard $\ZZ$-grading, one obtains the minimal $\ZZ^n$-graded free resolution, and hence the $\ZZ^n$-graded Betti numbers $\beta_{i,\sigma}(R/I)$ with $\sigma\in \ZZ^n$. Here $e_i$ denotes the $i^{th}$ standard basis vector in $\ZZ^n$.    
     
	A nonzero graded Betti number $\beta_{i,j}(R/I)$ of $R/I$ is called an \emph{extremal} Betti number if $\beta_{k,\ell}(R/I)=0$ for all $k\geq i$ and $\ell\geq j$ with $(k,\ell)\neq (i,j)$. It is easily seen that $R/I$ has a unique extremal Betti number if and only if $\beta_{p,p+r}(R/I)\neq 0$ where $p=\projdim (R/I)$ and $r=\reg (R/I)$. 
	
	\medskip
	We divide the rest of this section into two subsections devoted to the cases of monomial edge ideals and binomial edge ideals, respectively.   
	
	\subsection{(Monomial) edge ideals of graphs} 
	
	Let $R=\KK[x_1,\ldots,x_n]$ as above. 
	Recall that the (monomial) edge ideal of a graph $G$ on $n$ vertices is defined as 
	\[
	I(G)=(x_ix_j:\{i,j\}\in E(G)).
	\]    

	We gather some known results regarding the graded Betti numbers, 
	the projective dimension and the Castelnuovo-Mumford regularity of the (monomial) edge ideals of weakly chordal graphs in the next theorem. Here, for any $\sigma\subseteq V(G)$ we identify $\sigma$ and its characteristic vector in $\ZZ^n$. 
	
	\begin{Theorem}\label{weakly chordal-Betti}
		Let $G$ be a weakly chordal graph on $n$ vertices. Then the following statements hold:
		\begin{enumerate}
			\item[\em (a)] {\em(}\cite[Theorem~1.1]{K},\cite[Theorem~3.4]{MP}{\em)} $\beta_{|\sigma|-r, \sigma}(R/I(G))\neq 0$ if and only if there exists  $\MB\in \MS(G)$ with $V(\MB)=\sigma$ and $r=|\MB|$. 
			\item[\em (b)] \cite[Theorem~14]{W} $\reg (R/I(G))=\inm (G)$.
			\item[\em (c)] \cite[Theorem~7.7]{NV} 
			$\projdim (R/I(G))=d(G)$.
		\end{enumerate}
	\end{Theorem} 
	
	By Proposition~\ref{weakly chordal}, all of the statements in Theorem~\ref{weakly chordal-Betti} hold for any strongly biconvex graph. So, as an immediate consequence of this theorem and Corollary~\ref{d(H)}, we get the following recursive formula for the projective dimension of the (monomial) edge ideal of a strongly biconvex graph. 
	
	\begin{Corollary}\label{proj}
     Let $H=(X,Y)$ be a strongly biconvex graph with no isolated vertices and $e=\{x_q,y_{q'}\}$. Moreover, let $p_1=\projdim (S_1/I(H-V(B_e)))$ and $p_2=\projdim (S_2/I(H-x_q))$, where $S_1$ and $S_2$ are the polynomial rings over $\KK$ with variables correspond to vertices of $H-V(B_e)$ and $H-x_q$, respectively. Then 
     \[
     \projdim (S/I(H))=\max \{p_1+\deg_H(x_q)+\deg_H(y_{q'})-1, p_2\}.
     \]
	\end{Corollary}

\medskip 
Now, we construct a strongly biconvex graph $H_0=(X,Y)$, which plays role in the rest of this section, as follows. Let
\[
X=\{x_1,\ldots,x_{15}\}~~\text{and}~~Y=\{y_3,\ldots,y_{16}\} 
\] 
such that 
\[
m(1)=3~~,~~m(i)=i+1~~\text{for any}~~i=2,\ldots,15 
\]
and 
\[
M(1)=4~~,~~M(2)=8~~,~~M(j)=13~~\text{for any}~~j=3,\ldots,7,
\]
\[
M(8)=14~~,~~M(j)=16~~\text{for any}~~j=9,\ldots,15.
\]

The graph $H_0$ is depicted in Figure~\ref{fig H_0}. In the next theorem we investigate about the uniqueness of extremal Betti numbers of the (monomial) edge ideal of $H_0$.

\begin{Theorem}\label{H_0}
	Let $R=\KK[x_1,\ldots,x_{15},y_3,\ldots,y_{16}]$, and let $p=\projdim (R/I(H_0))$. Then $\beta_{p,p+4}(R/I(H_0))=0$. In particular, $R/I(H_0)$ does not have a unique extremal Betti number. 
\end{Theorem}

\begin{proof}
By Theorem~\ref{induced matching}, the set 
\[
\mathcal{M}(H_0)=\big \{ \{x_{1},y_{3}\}, \{x_{3},y_{5}\}, \{x_{8},y_{14}\}, \{x_{14},y_{16}\} \big \}
\] 
is a maximum induced matching for $H_0$. Then, $\reg (R/I(H_0))=4$, by Theorem~\ref{weakly chordal-Betti} part~(b). Therefore, the ``in particular" part follows once we prove $\beta_{p,p+4}(R/I(H_0))=0$. Suppose on the contrary that $\beta_{p,p+4}(R/I(H_0))\neq 0$. Thus, by Theorem~\ref{weakly chordal-Betti} part~(a), Lemma~\ref{ordered families} and Lemma~\ref{top matching}, there exists $\MB=\{B_1,B_2,B_3,B_4\}\in \mathcal{OS}(H)$ such that $|V(\MB)|=p+4$ and $d(H_0)=d(\MB)=p$. It is clear that $H_0-x_1$ is also a strongly biconvex graph. If $x_1\notin V(B_1)$, then $\MB\in \mathcal{OS}(H_0-x_1)$, a contradiction. Indeed, by Theorem~\ref{induced matching}, we have $\inm (H_0-x_1)=3$, 
while there are $4$ strongly disjoint complete subgraphs in $\MB$. So, suppose that $x_1\in V(B_1)$. Then by Theorem~\ref{B_e}, we may assume that $y_3\in V(B_1)$ and $B_1=B_e$ with $e=\{x_1,y_3\}$. Then $H_0-V(B_1)$ is strongly biconvex and we have $\{B_2,B_3,B_4\}\in \mathcal{OS}(H_0-V(B_1))$ and it is easily seen that 
\begin{equation}\label{d(H_0-V(B_1))}
d(H_0-V(B_1))=d(\{B_2,B_3,B_4\}).
\end{equation}
If $x_3\notin V(B_2)$, then it follows that $x_4,x_5,x_6,x_7\notin V(B_2)$, since otherwise the complete bipartite subgraph of $H_0$ on $V(B_2)\cup \{x_3\}$ together with $B_3$ and $B_4$ provide an element in $\mathcal{OS}(H_0-V(B_1))$, contradicting \eqref{d(H_0-V(B_1))}. On the other hand, by Theorem~\ref{induced matching},
$\inm (H_0-\{x_1,\ldots,x_7,y_3,y_4\})=2$, a contradiction to the fact that $\{B_2,B_3,B_4\}\in \mathcal{OS}(H_0-\{x_1,\ldots,x_7,y_3,y_4\})$. So, suppose that $x_3\in V(B_2)$. By Theorem~\ref{B_e}, we can take $\{B'_2,B_3,B_4\}\in \mathcal{OS}(H_0-V(B_1))$ where $B'_2$ is the complete bipartite graph on the vertex set $\{x_3,x_4\}\cup \{y_5,\ldots , y_{13}\}$ and such that 
\[
d(H_0-\{x_1,\ldots,x_7,y_3,\ldots, y_{13}\})=d(\{B_3,B_4\}).
\]  
If $x_8\notin V(B_3)$, then 
$\{B_3,B_4\}\in \mathcal{OS}(H_0-\{x_1,\ldots,x_8,y_3,\ldots, y_{13}\})$, a contradiction, since $\inm (H_0-\{x_1,\ldots,x_8,y_3,\ldots, y_{13}\})=1$. Therefore, suppose that $x_8\in V(B_3)$. Again, using Theorem~\ref{B_e}, we take $\{B'_3,B_4\}\in \mathcal{OS}(H_0-\{x_1,\ldots,x_7,y_3,\ldots, y_{13}\})$ where $B'_3$ is the complete bipartite subgraph on the vertices $\{x_8,\ldots,x_{13},y_{14}\}$ and 
\[
d(H_0-\{x_1,\ldots,x_{13},y_3,\ldots,y_{14}\})=d(\{B_4\}).
\]
The complete bipartite subgraph $B_4$ clearly consists of $3$ vertices, either $\{x_{14},x_{15},y_{16}\}$ or $\{x_{14},y_{15},y_{16}\}$. Finally, we get $d(H_0)=d(\{B_1,B'_2,B'_3,B_4\})$. But, we have \[
d(\{B_1,B'_2,B'_3,B_4\})=21,
\]
and hence $d(H_0)=p=21$. But the latter is a contradiction, since there is $\tilde{\MB}=\{\tilde{B_1},\tilde{B_2},\tilde{B_3}\}\in \mathcal{OS}(H_0)$ with $d(\tilde{\MB})=23>d(H_0)$ as follows: 
$V(\tilde{B_1})=\{x_2\}\cup \{y_3,\ldots,y_8\}$, 
$V(\tilde{B_2})=\{x_3,\ldots,x_8\}\cup \{y_9,\ldots,y_{13}\}$ and 
$V(\tilde{B_3})=\{x_9,\ldots,x_{13}\}\cup \{y_{14},y_{15},y_{16}\}$. 

Therefore, we deduce that $\beta_{p,p+4}(R/I(H_0))=0$, as desired.   
\end{proof} 

We would like to remark that arguments similar to our proof of Theorem~\ref{H_0} show that the projective dimension of $S/I(H_0)$ is indeed equal to $23$. 

\begin{figure}[h!]
	\centering
	{	\begin{tikzpicture}[scale = 1.3]
		\begin{scope}
		\draw (0,14) node[punkt] {} -- (2,12) node[punkt] {};
		\draw (0,14) node[punkt] {} -- (2,11) node[punkt] {};
		
		\draw (0,13) node[punkt] {} -- (2,12) node[punkt] {};
		\draw (0,13) node[punkt] {} -- (2,11) node[punkt] {};
		\draw (0,13) node[punkt] {} -- (2,10) node[punkt] {};
		\draw (0,13) node[punkt] {} -- (2,9) node[punkt] {};
		\draw (0,13) node[punkt] {} -- (2,8) node[punkt] {};
		\draw (0,13) node[punkt] {} -- (2,7) node[punkt] {};
		
		\draw (0,12) node[punkt] {} -- (2,11) node[punkt] {};
		\draw (0,12) node[punkt] {} -- (2,10) node[punkt] {};
		\draw (0,12) node[punkt] {} -- (2,9) node[punkt] {};
		\draw (0,12) node[punkt] {} -- (2,8) node[punkt] {};
		\draw (0,12) node[punkt] {} -- (2,7) node[punkt] {};
		\draw (0,12) node[punkt] {} -- (2,6) node[punkt] {};
		\draw (0,12) node[punkt] {} -- (2,5) node[punkt] {};
		\draw (0,12) node[punkt] {} -- (2,4) node[punkt] {};
		\draw (0,12) node[punkt] {} -- (2,3) node[punkt] {};
		\draw (0,12) node[punkt] {} -- (2,2) node[punkt] {};
		
		\draw (0,11) node[punkt] {} -- (2,10) node[punkt] {};
		\draw (0,11) node[punkt] {} -- (2,9) node[punkt] {};
		\draw (0,11) node[punkt] {} -- (2,8) node[punkt] {};
		\draw (0,11) node[punkt] {} -- (2,7) node[punkt] {};
		\draw (0,11) node[punkt] {} -- (2,6) node[punkt] {};
		\draw (0,11) node[punkt] {} -- (2,5) node[punkt] {};
		\draw (0,11) node[punkt] {} -- (2,4) node[punkt] {};
		\draw (0,11) node[punkt] {} -- (2,3) node[punkt] {};
		\draw (0,11) node[punkt] {} -- (2,2) node[punkt] {};
		
		\draw (0,10) node[punkt] {} -- (2,9) node[punkt] {};
		\draw (0,10) node[punkt] {} -- (2,8) node[punkt] {};
		\draw (0,10) node[punkt] {} -- (2,7) node[punkt] {};
		\draw (0,10) node[punkt] {} -- (2,6) node[punkt] {};
		\draw (0,10) node[punkt] {} -- (2,5) node[punkt] {};
		\draw (0,10) node[punkt] {} -- (2,4) node[punkt] {};
		\draw (0,10) node[punkt] {} -- (2,3) node[punkt] {};
		\draw (0,10) node[punkt] {} -- (2,2) node[punkt] {};
		
		\draw (0,9) node[punkt] {} -- (2,8) node[punkt] {};
		\draw (0,9) node[punkt] {} -- (2,7) node[punkt] {};
		\draw (0,9) node[punkt] {} -- (2,6) node[punkt] {};
		\draw (0,9) node[punkt] {} -- (2,5) node[punkt] {};
		\draw (0,9) node[punkt] {} -- (2,4) node[punkt] {};
		\draw (0,9) node[punkt] {} -- (2,3) node[punkt] {};
		\draw (0,9) node[punkt] {} -- (2,2) node[punkt] {};
		
		\draw (0,8) node[punkt] {} -- (2,7) node[punkt] {};
		\draw (0,8) node[punkt] {} -- (2,6) node[punkt] {};
		\draw (0,8) node[punkt] {} -- (2,5) node[punkt] {};
		\draw (0,8) node[punkt] {} -- (2,4) node[punkt] {};
		\draw (0,8) node[punkt] {} -- (2,3) node[punkt] {};
		\draw (0,8) node[punkt] {} -- (2,2) node[punkt] {};
		
		\draw (0,7) node[punkt] {} -- (2,6) node[punkt] {};
		\draw (0,7) node[punkt] {} -- (2,5) node[punkt] {};
		\draw (0,7) node[punkt] {} -- (2,4) node[punkt] {};
		\draw (0,7) node[punkt] {} -- (2,3) node[punkt] {};
		\draw (0,7) node[punkt] {} -- (2,2) node[punkt] {};
		\draw (0,7) node[punkt] {} -- (2,1) node[punkt] {};
		
		\draw (0,6) node[punkt] {} -- (2,5) node[punkt] {};
		\draw (0,6) node[punkt] {} -- (2,4) node[punkt] {};
		\draw (0,6) node[punkt] {} -- (2,3) node[punkt] {};
		\draw (0,6) node[punkt] {} -- (2,2) node[punkt] {};
		\draw (0,6) node[punkt] {} -- (2,1) node[punkt] {};
		\draw (0,6) node[punkt] {} -- (2,0) node[punkt] {};
		\draw (0,6) node[punkt] {} -- (2,-1) node[punkt] {};
		
		\draw (0,5) node[punkt] {} -- (2,4) node[punkt] {};
		\draw (0,5) node[punkt] {} -- (2,3) node[punkt] {};
		\draw (0,5) node[punkt] {} -- (2,2) node[punkt] {};
		\draw (0,5) node[punkt] {} -- (2,1) node[punkt] {};
		\draw (0,5) node[punkt] {} -- (2,0) node[punkt] {};
		\draw (0,5) node[punkt] {} -- (2,-1) node[punkt] {};
		
		\draw (0,4) node[punkt] {} -- (2,3) node[punkt] {};
		\draw (0,4) node[punkt] {} -- (2,2) node[punkt] {};
		\draw (0,4) node[punkt] {} -- (2,1) node[punkt] {};
		\draw (0,4) node[punkt] {} -- (2,0) node[punkt] {};
		\draw (0,4) node[punkt] {} -- (2,-1) node[punkt] {};
		
		\draw (0,3) node[punkt] {} -- (2,2) node[punkt] {};
		\draw (0,3) node[punkt] {} -- (2,1) node[punkt] {};
		\draw (0,3) node[punkt] {} -- (2,0) node[punkt] {};
		\draw (0,3) node[punkt] {} -- (2,-1) node[punkt] {};
		
		\draw (0,2) node[punkt] {} -- (2,1) node[punkt] {};
		\draw (0,2) node[punkt] {} -- (2,0) node[punkt] {};
		\draw (0,2) node[punkt] {} -- (2,-1) node[punkt] {};
		
		\draw (0,1) node[punkt] {} -- (2,0) node[punkt] {};
		\draw (0,1) node[punkt] {} -- (2,-1) node[punkt] {};
		
		\draw (0,0) node[punkt] {} -- (2,-1) node[punkt] {};
		
		\node [left] at (0,10) {$x_5$};
		\node [left] at (0,11) {$x_4$};
		\node [left] at (0,12) {$x_3$};
		\node [left] at (0,13) {$x_2$};
		\node [left] at (0,14) {$x_1$};
		\node [right] at (2,12) {$y_3$};
		\node [right] at (2,11) {$y_4$};
		\node [right] at (2,10) {$y_5$};
		\node [right] at (2,9) {$y_6$};
		\node [right] at (2,8) {$y_7$};
		\node [right] at (2,7) {$y_8$};
		
		\node [left] at (0,9) {$x_{6}$};
		\node [left] at (0,8) {$x_{7}$};
		\node [left] at (0,7) {$x_{8}$};
		\node [left] at (0,6) {$x_{9}$};
		\node [left] at (0,5) {$x_{10}$};
		\node [right] at (2,6) {$y_{9}$};
		\node [right] at (2,5) {$y_{10}$};
		\node [right] at (2,4) {$y_{11}$};
		\node [right] at (2,3) {$y_{12}$};
		
		\node [left] at (0,0) {$x_{15}$};
		\node [left] at (0,1) {$x_{14}$};
		\node [left] at (0,2) {$x_{13}$};
		\node [left] at (0,3) {$x_{12}$};
		\node [left] at (0,4) {$x_{11}$};
		\node [right] at (2,-1) {$y_{16}$};
		\node [right] at (2,0) {$y_{15}$};
		\node [right] at (2,1) {$y_{14}$};
		\node [right] at (2,2) {$y_{13}$};
		\end{scope}
		\end{tikzpicture}
		\caption{The graph $H_0$}
		\label{fig H_0}
	}
\end{figure}
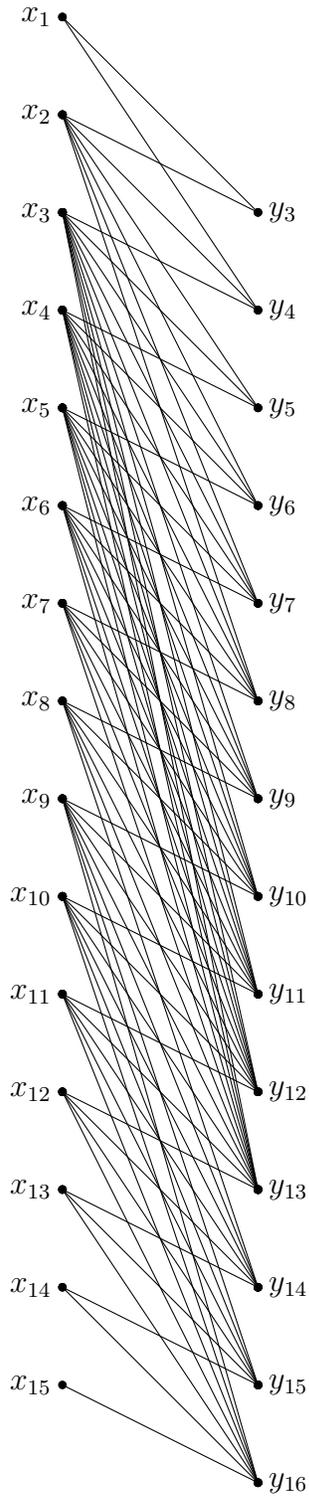

	\subsection{Binomial edge ideals of graphs} 
	
	Let $G$ be a graph with $n$ vertices, and let 
	$S=\KK[x_1,\ldots,x_n,y_1,\ldots,y_n]$ be a polynomial ring over a field $\KK$. The \emph{binomial edge ideal} of $G$, denoted by $J_G$, is defined as follows:
	\[
	J_G=(x_iy_j-x_jy_i: i<j , \{i,j\}\in E(G)).
	\]
	
	Let $<$ be the lexicographic order on $S$ induced by $x_1>\cdots >x_n>y_1> \cdots >y_n$. The following theorem determines the relationship between the regularity and the projective dimension of $J_G$ and its initial ideal in terms of the lexicographic order. We use this relationship later in this section.
	
	\begin{Theorem}\label{Conca}
		{\em (}\cite[Corollary~2.7]{CV}, \cite[Theorem~2.1]{HHHKR}{\em )}
		Let $G$ be a graph. Then:
		\begin{enumerate}
	    \item[\em (a)] $\reg (S/J_G)=\reg (S/\ini_< (J_G))$;
		\item[\em (b)] $\projdim (S/J_G)=\projdim (S/\ini_< (J_G))$. 
		\end{enumerate}
	\end{Theorem}

	In \cite{HHHKR}, those graphs $G$ whose binomial edge ideals admit a quadratic Gr\"obner basis, and hence a quadratic initial ideal, were determined. Indeed, it was shown that the aforementioned binomial generators of $J_G$ provide a quadratic Gr\"obner basis for $J_G$ if and only if $G$ is a closed graph (see \cite[Theorem~1.1]{HHHKR}). A \emph{closed} graph $G$ is a graph which has a labeling of its vertices for which the following property holds: for all edges $\{i,j\}$ and $\{k,\ell\}$ with $i<j$ and $k<\ell$, one has $\{j,\ell\}\in E(G)$ if $i=k$, and $\{i,k\}\in E(G)$ if $j=\ell$. There are several combinatorial characterizations for closed graphs, like \cite[Theorem~2.2]{EHH} where it was shown that $G$ is closed if and only if the vertices of $G$ can be labeled such that all of the cliques (i.e. maximal complete subgraphs) of $G$ are intervals. 
	
	If $G$ is a closed graph, then we have 
	\[
	\ini_<(J_G)=(x_iy_j:i<j , \{i,j\}\in E(G)).
	\]
	This shows that the initial ideal of the binomial edge ideal of a closed graph with $n\geq 2$ vertices is in fact the (monomial) edge ideal of a bipartite graph on the vertex set $X\cup Y$ with $X=\{x_1,\ldots,x_{n-1}\}$ and $Y=\{y_2,\ldots,y_n\}$, and the edge set 
	\[
	\big{\{} \{x_i,y_j\} : i<j , \{i,j\}\in E(G) \big{\}}
	\]  
	which has no isolated vertex. We call this graph the \emph{initial graph} of $G$, and following~\cite{SK}, we denote it by $\ini (G)$. Indeed, we have 
	\[
	I(\ini (G))=\ini_<(J_G).
	\]  
	
	The following proposition shows that closed graphs imply a subclass of strongly biconvex graphs via their initials.
	
	\begin{Proposition}\label{closed-initial}
		Let $G$ be a closed graph with at least two vertices. Then $\ini (G)$ is a strongly biconvex graph.  
	\end{Proposition}  

\begin{proof}
Since $G$ is closed, there exists a labeling for its vertices, like $\{1,\ldots,n\}$, such that the maximal cliques of $G$ are intervals. By the definitinon of $\ini (G)$, condition~(1) in the Definition~\ref{strongly biconvex} clearly holds. Now, let $\{x_i,y_j\}$ be an edge of $\ini (G)$ and let $i<r<j$. 
It follows that $\{i,j\}\in E(G)$, and hence is contained in a maximal clique which is labeled by an interval. Therefore, $\{i,r\}\in E(G)$ and $\{r,j\}\in E(G)$. By the construction of $\ini (G)$, then we deduce that $\{x_i,y_r\}$ and $\{x_r,y_j\}$ are both edges of $\ini (G)$, and hence condition~(2) in the Definition~\ref{strongly biconvex} hold. Thus, $\ini (G)$ is a strongly biconvex graph.  
\end{proof}
	
Note that not all strongly biconvex graphs are initial graph of a closed graph. For instance, the graph shown in Figure~\ref{strongly biconvex example} is not the initial graph of any closed graph, as it has odd number of vertices. 

\medskip	
Now we construct a closed graph on $17$ vertices. Let $G_0$ be the closed graph on the vertex set $\{1,\ldots,17\}$ given by the maximal cliques $[1,3]$, $[2,5]$, $[3,9]$, $[4,14]$, $[9,15]$ and $[10,17]$. Using this graph, we give a negative answer to \cite[Question~1]{AH} in the following theorem. 

\begin{Theorem}\label{nonunique closed}
 Let $S=\KK[x_1,\ldots,x_{17},y_1,\ldots,y_{17}]$ and let $p=\projdim (S/J_{G_0})$. Then $\beta_{p,p+5}(S/J_{G_0})=0$. In particular, $S/J_{G_0}$ does not have a unique extremal Betti number. 
\end{Theorem}  	
	
\begin{proof}
First we relabel the vertex set of the graph $H_0$ by replacing $x_i$ and $y_i$ with $x_{i+1}$ and $y_{i+1}$, respectively. We denote the obtained graph by $H'_0$. Then, define a new graph $H''_0$ with 
\[
V(H''_0)=V(H'_0)\cup \{x_1,y_2,y_3\}
\] 
and 
\[
E(H''_0)=E(H'_0)\cup \big{\{} \{x_1,y_2\},\{x_1,y_3\}, \{x_2,y_3\} \big{\}}. 
\]
Then it is easy to see that $\ini (G_0)=H''_0$, and hence $H''_0$ is a strongly biconvex graph. So, by Theorem~\ref{Conca}, we have $p=\projdim (S/J_{G_0})=\projdim (S/I(H''_0))$ and 
$\reg (S/J_{G_0})=\reg (S/I(H''_0))=5$. The last equality follows from Theorem~\ref{induced matching}, since $\{x_1,y_2\}$, $\{x_2,y_4\}$, $\{x_4,y_6\}$, $\{x_9,y_{15}\}$ and $\{x_{15},y_{16}\}$ provide a maximum induced matching for $H''_0$. Now, suppose on the contrary that $\beta_{p,p+5}(S/J_{G_0})\neq 0$. Then it follows from \cite[Corollary~3.3.3]{HH} that $\beta_{p,p+5}(S/I(H''_0))\neq 0$. Thus, by Theorem~\ref{weakly chordal-Betti} part~(a), there exists $\MB=\{B_1,\ldots,B_5\}\in \mathcal{OS}(H''_0)$ such that $d(H''_0)=d(\MB)=p$ and $|V(\MB)|=p+5$. If $x_1\notin V(B_1)$, then $V(\MB)\subseteq V(H''_0-\{x_1,y_2\})$ which is a contradiction, because $\inm (H''_0-\{x_1,y_2\})=4$. So, suppose that $x_1\in V(B_1)$. Then by Theorem~\ref{B_e}, there exists $\MB'=\{B'_1,B_2,B_3,B_4,B_5\}\in \mathcal{OS}(H''_0)$ with $B'_1=B_e$ and $d(\MB')=d(H''_0)=p$ where $e=\{x_1,y_2\}$. Therefore, $V(\MB'\setminus \{B'_1\})\subseteq V(H'_0)$ and moreover, we have $\MB'\setminus \{B'_1\}\in \mathcal{OS}(H'_0)$ and $d(H'_0)=d(\MB'\setminus \{B'_1\})$. By Theorem~\ref{weakly chordal-Betti} part~(c) we have $d(H'_0)=\projdim (S/I(H'_0))$. Again using Theorem~\ref{weakly chordal-Betti}, we get $\beta_{q-4,q}(S/I(H'_0))\neq 0$ where $q=|V(\MB'\setminus \{B'_1\})|$ and $d(H'_0)=q-4$. Since $H'_0$ and $H_0$ are isomorphic, it follows that $\beta_{q-4,q}(S/I(H_0))\neq 0$, a contradiction to Theorem~\ref{H_0}. Therefore, we get $\beta_{p,p+5}(S/I(H''_0))=0$. 
\end{proof}	

Next, we construct an infinite family of closed graphs whose binomial edge ideals do not have a unique extremal betti number. For this purpose, we fix the following notation. If $G_1$ and $G_2$ are two closed graphs on disjoint sets of vertices (with the desired labeling) $\{1,\ldots,n_1\}$ and $\{n_1+1,\ldots, n_2\}$, respectively, then by identifying the two vertices $n_1$ and $n_1+1$ we get a new graph which is clearly closed as well.  
Now, for any $t$, by applying the above procedure on $t$ disjoint copies of the closed graph $G_0$, we get a new closed graph on $16t+1$ vertices and we denote it by $G_{0,t}$. 
	
The next corollary discusses non-uniqueness of the extremal betti numbers of the binomial edge ideals of this family of graphs. Here, $S$ is an appropriate polynomial ring with the desired number of variables.  	
	
\begin{Corollary}\label{infinite family}
	Let $t\geq 1$ and $p=\projdim (S/J_{G_0})$. Then $\beta_{tp,tp+5t}(S/J_{G_{0,t}})=0$. In particular, $S/J_{G_{0,t}}$ does not have a unique extremal betti number. 
\end{Corollary}	

\begin{proof}
It is easy to see that $\ini (G_{0,t})$ is a graph with $t$ connected components where each of them is a copy of $H''_0$ with the desired labeling according to the labeling of $G_{0,t}$. Let $H''_{0,1},\ldots,H''_{0,t}$ be those copies of $H''_0$. Since $H''_{0,\ell}$'s are on disjoint sets of vertices, it follows from \cite[Lemma~2.1]{JK} that the minimal graded free resolution of $S/I(\ini (G_{0,t}))$ is obtained from the tensor product of the minimal graded free resolutions of $S_{\ell}/I(H''_{0,\ell})$'s, where $S_{\ell}$ is the polynomial ring over $\KK$ with suitable variables. Hence we have
\begin{equation}
\beta_{tp,tp+5t}(S/I(\ini (G_{0,t})))=\sum_{\substack{
		i_1+\cdots +i_{\ell}=pt \\
        j_1+\cdots +j_{\ell}=5t
}} \prod_{\ell=1}^t \beta_{i_{\ell},i_{\ell}+j_{\ell}} (S_{\ell}/I(H''_{0,\ell})).
\nonumber
\end{equation} 
Obviously, $\beta_{i_{\ell},i_{\ell}+j_{\ell}}(S_{\ell}/I(H''_{0,\ell}))=0$ 
for any $\ell$ with $i_{\ell}> p$ and $j_{\ell}> 5$. Thus, we have  $i_1=\cdots=i_t=p$ and $j_1=\cdots=j_t=5$, and hence 
\[
\beta_{tp,tp+5t}(S/I(\ini (G_{0,t})))=\prod_{\ell=1}^t \beta_{p,p+5} (S_{\ell}/I(H''_{0,\ell}))=0
\]
by Theorem~\ref{nonunique closed}. Therefore, by \cite[Corollary~3.3.3]{HH} we have $\beta_{tp,tp+5t}(S/J_{G_{0,t}})=0$, as desired. Then, the ``in particular" part follows from Theorem~\ref{Conca} which implies that 
\[
\reg (S/J_{G_{0,t}})=\reg (S/I(\ini (G_{0,t})))=5t
\] 
and 
\[
\projdim (S/J_{G_{0,t}})=\projdim (S/I(\ini (G_{0,t})))=tp.
\]       	
\end{proof}

	

\par \textbf{Acknowledgments:} 
The research of the first author was in part supported by a grant from IPM (No. 98130013). The research of the second author was in part supported by a grant from IPM (No. 98050212).

\end{document}